\documentclass[10pt]{amsart}
\usepackage{amssymb,amsbsy,amsmath,amsfonts,amssymb}
\usepackage{latexsym,euscript,exscale}

\title{$\alpha$-minimal Banach spaces}
\author {Christian Rosendal}
\address{Department of Mathematics, Statistics, and Computer Science (M/C 249)\\
University of Illinois at Chicago\\
851 S. Morgan St.\\
Chicago, IL 60607-7045\\
USA}
\email{rosendal.math@gmail.com}
\urladdr{http://www.math.uic.edu/$_~$rosendal}

\newcommand {\N}{\mathbb N}

\newcommand {\R}{\mathbb R}

\newcommand{\I}{\mathbb I}

\newcommand {\D}{\mathbb D}

\newcommand{\norm}[1]{\lVert#1\rVert}

\newcommand{\om}{\omega}

\newcommand{\con}{\;\hat{}\;}

\newcommand{\tom} {\emptyset}

\newcommand{\saa}{\Rightarrow}
\newcommand{\equi}{\Leftrightarrow}

\newcommand{\til}{\rightarrow}

\newcommand{\Lim}[1]{\mathop{\longrightarrow}\limits_{#1}}

\newcommand {\del}{ \; \big| \;}

\newcommand {\go} {\mathfrak}
\newcommand {\ku} {\mathcal}

\newcommand {\e} {\exists}
\renewcommand {\a} {\forall}

\newtheorem{thm}{Theorem}

\newtheorem{lemme}[thm]{Lemma}

\newtheorem{prop} [thm] {Proposition}
\newtheorem{defi} [thm] {Definition}

\newtheorem{prob}[thm]{Problem}

\begin{document}

\thanks{The initial research for this article was done while the author was visiting V. Ferenczi at the University of S\~ao Paulo, Brazil, with the support of FAPESP.  The author's research was likewise supported by  NSF grants DMS 0901405 and DMS 0919700}



\maketitle

\begin{abstract}A Banach space $\ku W$ with a Schauder basis is said to be {\em $\alpha$-minimal} for some $\alpha<\om_1$ if, for any two block subspaces $\ku Z, \ku Y\subseteq \ku W$, the Bourgain embeddability index of $\ku Z$ into $\ku Y$ is at least $\alpha$.

We prove a dichotomy that characterises when a Banach space has an $\alpha$-minimal subspace, which contributes to the ongoing project, initiated by W. T. Gowers, of classifying separable Banach spaces by identifying characteristic subspaces.
\end{abstract}

\tableofcontents

\section{Introduction}
Suppose $\ku W$ is a separable, infinite-dimensional Banach space. We say that $\ku W$ is {\em minimal} if $\ku W$ isomorphically embeds into any infinite-dimensional subspace $\ku Y\subseteq \ku W$ (and write $\ku W\sqsubseteq \ku Y$ to denote that $\ku W$ embeds into $\ku Y$). The class of Banach spaces without minimal subspaces was studied by V. Ferenczi and the author in \cite{minimal}, extending work of W. T. Gowers \cite{gowers} and A. M. Pelczar \cite{anna}, in which a dichotomy was proved characterising the presence of minimal subspaces in an arbitrary infinite-dimensional Banach space. 

The dichotomy hinges on the notion of {\em tightness}, which we can define as follows. Assume that $\ku W$ has a Schauder basis $(e_n)$ and suppose $\ku Y\subseteq \ku W$ is a subspace. We say that $\ku Y$ is {\em tight} in the basis $(e_n)$ for $\ku W$ if there are successive finite intervals of $\N$,
$$
I_0<I_1<I_2<\ldots\subseteq \N,
$$
such that for any isomorphic embedding $T\colon \ku Y\til \ku W$, if $P_{I_m}$ denotes the canonical projection of $\ku W$ onto $[e_n]_{n\in I_m}$, then 
$$
\liminf_{m\til \infty}\|P_{I_m}T\|>0.
$$
Alternatively, this is equivalent to requiring that whenever $A\subseteq \N$ is infinite, there is no embedding of $\ku Y$ into $[e_n\del n\notin \bigcup_{m\in A}I_m]$.
Also, the basis $(e_n)$ is {\em tight} if any infinite-dimensional subspace $\ku Y\subseteq \ku W$ is tight in $(e_n)$ and a space is {\em tight} in case it has a tight basis. We note that if $\ku W$ is tight, then so is any shrinking basic sequence in $\ku W$.

Tightness is easily seen to be an obstruction to minimality, in the sense that a tight space cannot contain a minimal subspace. In \cite{minimal} the following converse is proved:  any infinite-dimensional Banach space contains either a minimal or a tight subspace.

J. Bourgain introduced in \cite{bourgain} an ordinal index that gives a quantitative measure of how much one Banach space with a basis embeds into another. Namely, suppose $\ku W$ is a space with a  Schauder basis $(e_n)$ and $\ku Y$ is any Banach space. We let $T((e_n),\ku Y,K)$ be the tree of all finite sequences $(y_0,y_1,\ldots,y_k)$ in $\ku Y$, including the empty sequence $\tom=(\;)$, such that
$$
(y_0,\ldots,y_k)\sim_K(e_0,\ldots,e_k).
$$
Here, whenever $(x_i)$ and $(y_i)$ are sequences of the same (finite or infinite) length in Banach spaces $\ku X$ and $\ku Y$, we write 
$$
(x_i)\sim_K(y_i)
$$
if for all $a_{0},\ldots,a_{k}\in \R$
$$
\frac 1K\Big\|\sum_{i=0}^ka_{i}x_{i}\Big\|\leqslant \Big\|\sum_{i=0}^ka_{i}y_{i}\Big\|\leqslant K\Big\|\sum_{i=0}^ka_{i}x_{i}\Big\|.
$$
We notice that $T((e_n),\ku Y,K)$ is {\em ill-founded}, i.e., admits an infinite branch, if and only if  $\ku W=[e_n]$ embeds with constant $K$ into $\ku Y$. 

The {\em rank function} $\rho_T$ on a {\em well-founded} tree $T$, i.e., without infinite branches, is defined by $\rho_T(s)=0$ if $s\in T$ is a terminal node and 
$$
\rho_T(s)=\sup\big\{\rho_T(t)+1\del s\prec t, \; t\in T\big\}
$$
otherwise. Then, the {\em rank} of $T$ is defined by
$$
{\rm rank}(T)=\sup\big\{\rho_T(s)+1\del s\in T\big\},
$$ 
whence ${\rm rank}(T)=\rho_T(\tom)+1$ if $T$ is non-empty.
Moreover, if $T$ is ill-founded, we let ${\rm rank}(T)=\infty$, with the stipulation that $\alpha<\infty$ for all ordinals $\alpha$.

Then, ${\rm rank}\big(T((e_n),\ku Y,K)\big)$ measures the extent  to which $\ku W=[e_n]$ $K$-embeds into $\ku Y$ and we therefore define the {\em embeddability rank} of $\ku W=[e_n]$ into 
$\ku Y$ by
$$
{\rm Emb}((e_n),\ku Y)=\sup_{K\geqslant 1}{\rm rank}\big(T((e_n),\ku Y,K)\big).
$$
Since $(e_n)$ is a basic sequence, there is for any $K\geqslant 1$ a sequence $\Delta=(\delta_n)$ of positive real numbers, such that if $y_n,z_n\in \ku Y$, $\norm{y_n-z_n}<\delta_n$ and $(y_0,\ldots,y_k)\sim_K(e_0,\ldots,e_k)$, then also $(z_0,\ldots,z_k)\sim_{K+1}(e_0,\ldots,e_k)$. Therefore, to calculate the embeddability rank, ${\rm Emb}((e_n),\ku Y)$, it suffices to consider the trees of all finite sequences $(y_0,\ldots,y_k)$ with $(y_0,\ldots,y_k)\sim_K(e_0,\ldots,e_k)$, where, moreover, we require the $y_n$ to belong to some fixed dense subset of $\ku Y$. We shall use this repeatedly later on, where we replace $\ku Y$ by a dense subset of itself. 
This comment also implies that ${\rm Emb}((e_n),\ku Y)$ is either $\infty$, if $\ku W\sqsubseteq \ku Y$, or an ordinal $<{\rm density}(\ku Y)^+$, if $\ku W\not\sqsubseteq \ku Y$. In particular, if $\ku Y$ is separable, then ${\rm Emb}((e_n),\ku Y)$ is either $\infty$ or a countable ordinal. 
Also, note that the embeddability rank depends not only on the space $\ku W$, but also on the basis $(e_n)$. However, if $\ku Y$ is separable and $\ku W\not\sqsubseteq\ku Y$, then by the Boundedness Theorem for coanalytic ranks (see \cite{kechris}), the supremum of ${\rm Emb}((e_n),\ku Y)$ over all bases $(e_n)$ for $\ku W$ is a countable ordinal. 
In case ${\rm Emb}((e_n),\ku Y)\geqslant \alpha$, we say that $\ku W=[e_n]$ {\em $\alpha$-embeds} into $\ku Y$.

Since minimality is explicitly expressed in terms of embeddability, it is natural to combine it with Bourgain's embeddability index in the following way. 
\begin{defi}Let $\alpha$ be a countable ordinal.
A Banach space $\ku W$ with a Schauder basis $(e_n)$ is {\em $\alpha$-minimal} if any block subspace $\ku Z=[z_n]\subseteq \ku W$ $\alpha$-embeds into any infinite-dimensional subspace $\ku Y\subseteq \ku W$.
\end{defi}
It is easy to check that if $\ku W=[e_n]$ is a space with a basis and $\ku X=[x_n]$ and $\ku Y=[y_n]$ are block subspaces of $\ku W$ such that $x_n\in \ku Y$ for all but finitely many $n$, which we denote by $\ku X\subseteq^*\ku Y$,  then if $\ku Y$ is $\alpha$-minimal, so is $\ku X$. In particular, $\alpha$-minimality is preserved by passing to block subspaces.

Similarly, we can combine tightness with the embeddability index.
\begin{defi}
Let $\alpha$ be a countable ordinal and $\ku W$ a Banach space with a Schauder basis $(e_n)$. We say that $\ku W=[e_n]$ is $\alpha$-tight if for any block basis $(y_n)$ in $\ku W$ there is a sequence of intervals of $\N$,
$$
I_0<I_1<I_2<\ldots\subseteq \N
$$
such that for any infinite set $A\subseteq \N$,
$$
{\rm Emb}\big((y_n),[e_n\del n\notin \bigcup_{j\in A}I_j]\big)\leqslant\alpha.
$$
In other words, if $\ku Y=[y_n]$ $(\alpha+1)$-embeds into some subspace $\ku Z\subseteq \ku W$, then 
$$
\liminf_{k\til \infty}\|P_{I_k}|_{\ku Z}\|>0.
$$
\end{defi}
Again, it is easy to see that if $\ku W=[e_n]$ is $\alpha$-tight, then so is any block subspace of $\ku W$.
Also, if $\ku W=[e_n]$ is $\alpha$-tight, then no block subspace, $\ku Y=[y_n]$,  is $\beta$-minimal for $\alpha<\beta$. And, if $\ku Y=[y_n]$ is minimal, then $\ku Y=[y_n]$ is $\alpha$-minimal for any $\alpha<\om_1$. It follows from this that if $\ku W=[e_n]$ is $\alpha$-tight, then $\ku W=[e_n]$ admits no minimal block subspaces, and thus, as any infinite-dimensional subspace contains a block subspace up to a small perturbation, $\ku W$ contains no minimal subspaces either.

Our first result says that tightness can be reinforced to $\alpha$-tightness.
\begin{thm}\label{tightness}
Let $\ku W$ be a Banach space with a Schauder basis  and having no minimal subspaces. Then there is a block subspace $\ku X=[x_n]$ that is  $\alpha$-tight for some countable ordinal $\alpha$.
\end{thm}

Our main results, however, provides us with more detailed structural information.
\begin{thm}\label{main5}
Let $\ku W$ be Banach space with a Schauder basis and suppose $\alpha<\om_1$. 
Then there is a block subspace  $\ku X=[x_n]\subseteq \ku W$ that is either $\om\alpha$-tight or $(\alpha+1)$-minimal.
\end{thm}

Finally, combining Theorems \ref{tightness} and \ref{main5}, we have the following refinement of Theorem \ref{tightness}.
\begin{thm}\label{principal}
Let $\ku W$ be a Banach space with a Schauder basis. Then $\ku W$ has a minimal subspace or a block subspace $\ku X=[x_n]\subseteq \ku W$ that  is $\alpha$-minimal and $\om\alpha$-tight for some countable ordinal $\alpha$.
\end{thm}

\begin{proof}
Suppose that $\ku W$ has no minimal subspace and pick by Theorem \ref{tightness} some block subspace $\ku W_0\subseteq \ku W$ that is $\beta$-tight for some $\beta<\om_1$. So no block subspace of $\ku W_0$ is $(\beta+1)$-minimal. Let now $\alpha$ be the supremum of all ordinals $\gamma$ such that $\ku W_0$ is saturated with $\gamma$-minimal block subspaces and pick a block subspace $\ku W_1\subseteq \ku W_0$ not containing any $(\alpha+1)$-minimal subspace. 

We claim that $\ku W_1$ contains a $\alpha$-minimal block subspace $\ku W_\infty$. If $\alpha$ is a successor ordinal, this is obvious, so suppose instead that $\alpha$ is a limit. Then we can find ordinals $\gamma_2<\gamma_3< \ldots$ with supremum $\alpha$. We then inductively choose block subspaces $\ku W_1\supseteq \ku W_2\supseteq \ku W_3\supseteq \ldots$ such that $\ku W_n$ is $\gamma_n$-minimal. Letting $\ku W_\infty\subseteq \ku W_1$ be a block subspace such that $\ku W_\infty\subseteq ^*\ku W_n$ for all $n$, we see that $\ku W_\infty$ is $\gamma_n$-minimal for all $n$, which means that for any block sequence $(z_m)\subseteq \ku W_\infty$ and infinite-dimensional subspace $\ku Y\subseteq \ku W_\infty$, we have 
$$
{\rm Emb}\big((z_m),\ku Y\big)\geqslant \gamma_n
$$
for all $n$, whence ${\rm Emb}\big((z_m),\ku Y\big)\geqslant \sup_{n}\gamma_n=\alpha$. So $\ku W_\infty$ is $\alpha$-minimal and so are its subspaces.

Now, $\ku W_\infty$ has no $(\alpha+1)$-minimal subspace, so, by Theorem \ref{main5}, $\ku W_\infty$ contains an $\om\alpha$-tight block subspace $\ku X$, which simultaneously is $\alpha$-minimal.
\end{proof}

Since any two Banach spaces of the same finite dimension are isomorphic, one easily sees that any space $\ku W$ with a Schauder basis $(e_n)$ is $\om$-minimal. On the other hand, in \cite{minimal}, a space $\ku W=[e_n]$ is defined to be {\em tight with constants} if for any block subspace $\ku Y=[y_n]$ there are intervals $I_0<I_1<I_2<\ldots$ such that for any integer constant $K$, 
$$
[y_n]_{n\in I_K}\not\sqsubseteq_K[e_n]_{n\notin I_K}.
$$
In this case, it follows that for any infinite set $A\subseteq \N$ and any $K\in A$, 
$$
{\rm rank}\big(T((y_n),[e_n\del n\notin \bigcup_{j\in A}I_j], K)\big)\leqslant \max I_K,
$$
and hence
$$
{\rm Emb}((y_n),[e_n\del n\notin \bigcup_{j\in A}I_j])=\sup_{K\in A}{\rm rank}\big(T((y_n),[e_n\del n\notin \bigcup_{j\in A}I_j], K)\big)\leqslant \om.
$$
So, if $\ku W=[e_n]$ is tight with constants, we see that $\ku W=[e_n]$ is $\om$-tight and $\om$-minimal.

Following \cite{minimal}, we also define a space $\ku W$ to be {\em locally minimal} if there is a constant $K\geqslant 1$ such that $\ku W$ is $K$-crudely finitely representable in any infinite-dimensional subspace, i.e., if for any finite-dimensional $\ku F\subseteq \ku W$ and infinite-dimensional $\ku Y\subseteq\ku  W$, $\ku F\sqsubseteq_K\ku Y$. Let us first see local minimality in terms of $\alpha$-minimality.
\begin{prop}
Suppose $\ku W$ is a locally minimal Banach space with a Schauder basis $(e_n)$. Then $\ku W=[e_n]$ is $\om^2$-minimal.
\end{prop}

\begin{proof}Let $K$ be the constant of local minimality.
For any infinite-dimensional subspace $\ku Y\subseteq \ku W$, block sequence $(w_i)\subseteq \ku W$  and $\alpha<\om^2$, we need to show that ${\rm Emb}((w_i),\ku Y)>\alpha$. So choose $n$ such that $\alpha<\om\cdot n$ and find some constant $C$ such that if $x_1<\ldots<x_n$ and $y_1<\ldots<y_n$ are finite block sequences of $(e_i)$ such that $\frac 1K\|x_i\|\leqslant \|y_i\|\leqslant K\|x_i\|$, then $(x_i)\sim_C(y_i)$. We claim that
$$
{\rm rank}\big(T((w_i),\ku Y,2C)\big)\geqslant \om\cdot n.
$$
To see this, find some block subspace $\ku X$ such that $\ku X\sqsubseteq_2\ku Y$. It suffices to prove that 
$$
{\rm rank}\big(T((w_i),\ku X,C)\big)\geqslant \om\cdot n.
$$

Let $k_1$ be given. We shall see that $\tom$ has rank $\geqslant \om(n-1)+k_1-1$ in $T((w_i),\ku X,C)$. So choose by local $K$-minimality some $z_0,\ldots,z_{k_1-1}\in \ku X$ such that
$$
(w_0,\ldots,w_{k_1-1})\sim_K(z_0,\ldots,z_{k_1-1}).
$$
It then suffices to show that $(z_0,\ldots,z_{k_1-1})$ has rank $\geqslant \om(n-1)$ in $T((w_i),\ku X,C)$, or, equivalently, that for any $k_2$, it has rank $\geqslant \om(n-2)+k_2-1$. So choose $z_{k_1},\ldots, z_{k_1+k_2-1}$ in $\ku X$ with support after all of $z_0,\ldots,z_{k_1-1}$ such that
$$
(w_{k_1},\ldots,w_{k_1+k_2-1})\sim_K(z_{k_1},\ldots,z_{k_1+k_2-1}).
$$ 
Again, it suffices to show that 
$$
(z_0,\ldots, z_{k_1-1},z_{k_1},\ldots,z_{k_1+k_2-1})
$$
has rank $\geqslant \om(n-2)$ in $T((w_i),\ku X,C)$. Et cetera.

Eventually, we will have produced 
$$
z_0,\ldots,z_{k_1-1}<z_{k_1},\ldots,z_{k_1+k_2-1}<\ldots<z_{k_1+\ldots+k_{n-1}},\ldots, z_{k_1+\ldots+k_{n}-1}
$$
such that for each $l$, 
$$
(w_{k_1+\ldots+k_{l-1}},\ldots,w_{k_1+\ldots+k_l-1})\sim_K(z_{k_1+\ldots+k_{l-1}},\ldots,z_{k_1+\ldots+k_l-1}).
$$
Since we have chosen the successive sections of $(z_i)$ successively on the basis, we have, by the choice of $C$, that
$$
(w_0,\ldots,w_{k_1+\ldots+k_{n}-1})\sim_C(z_0,\ldots,z_{k_1+\ldots+k_{n}-1}),
$$
 whereby $(z_0,\ldots,z_{k_1+\ldots+k_{n}-1})\in T((w_i),\ku X,C)$ and hence  has rank $\geqslant 0=\om(n-n)$ in  $T((w_i),\ku X,C)$. This finishes the proof.
\end{proof}

In \cite{minimal}, another dichotomy was proved stating that any infinite-dimensional Banach space contains a subspace with a basis  that is either tight with constants or is locally minimal. In particular, we have the following dichotomy.
\begin{thm}[V. Ferenczi and C. Rosendal \cite{minimal}]
Any infinite-dimensional Banach space contains an infinite-dimensional subspace with a basis that is either $\om$-tight or is $\om^2$-minimal.
\end{thm}

One problem that remains open is to exhibit spaces that are $\alpha$-minimal and $\om\alpha$-tight for unbounded $\alpha<\om_1$. We are not aware of any construction in the literature that would produce this, but remain firmly convinced that such spaces must exist, since otherwise there would be a universal $\beta<\om_1$ such that any Banach space would either contain a minimal subspace or a $\beta$-tight subspace, which seems unlikely.

\begin{prob}
Show that there are $\alpha$-minimal, $\om\alpha$-tight spaces for unboundedly many $\alpha<\om_1$.
\end{prob}

Out main result, Theorem \ref{principal}, allows us to refine the classification scheme developed in \cite{gowers} and \cite{minimal}, by further differentiating the class of tight spaces into $\alpha$-minimal, $\om\alpha$-tight for $\alpha<\om_1$. Currently, the most interesting direction for further results would be to try to distinguish between different classes of minimal spaces, knowing that these pose particular problems for applying Ramsey Theory.

Apart from some basic facts about Schauder bases, the main tools of our paper originate in descriptive set theory for which our general reference is the  book by A. S. Kechris \cite{kechris}.  In particular, we follow his presentation of trees and games, except that we separate a game from its winning condition and thus talk about players having a strategy to {\em play  in a certain set}, rather than having a strategy to win.

\section{Setup}\label{discrete}

For the proof of Theorem \ref{main5}, we will need to replace Banach spaces with the more combinatorial setting of normed vector space over countable fields, which we will be using throughout the paper (cf. \cite{exact}).
So suppose $\ku W$ is a Banach space with a Schauder basis $(e_n)$. By a standard Skolem hull construction, we find a countable subfield $\go F\subseteq \R$ such that for any $\go F$-linear combination
$\sum_{n=0}^ma_ne_n$,
the norm $\|\sum_{n=0}^ma_ne_n\|$ belongs to $\go F$.  Let also $W$ be the countable-dimensional $\go
F$-vector space with basis $(e_n)$. In the following, we shall exclusively consider the $\go F$-vector space structure of $W$, and thus subspaces etc. refer to $\go F$-vector subspaces.
We equip $W$ with the discrete topology,
whereby any subset is open, and equip its countable power $W^\N$ with the
product topology. Since $W$ is a countable discrete set, $W^\N$ is a Polish, i.e., separable and completely metrisable,
space. Notice that a basis for the topology on $W^\N$ is given by sets of
the form
$$
N(x_0,\ldots,x_k)=\{(y_n)\in W^\N\del y_0=x_0\;\&\;\ldots\;\&\;y_k=x_k\},
$$
where $x_0,\ldots,x_k\in W$. Henceforth, we let $x,y,z,v$ be variables for {\em non-zero}
elements of $W$. If $x=\sum a_ne_n\in W$, we define the {\em support} of $x$ to be the finite, non-empty set ${\rm supp}(x)=\{n\del a_n\neq
0\}$ and set for $x,y\in W$,
$$
x<y\equi \a n\in {\rm supp}(x)\; \a m\in {\rm supp}(y)\;\; n<m.
$$
Similarly, if $k$ is a natural number, we set
$$
k<x\equi \a n\in {\rm supp}(x)\;\;k<n.
$$
Analogous notation is used for finite subsets of $\N$ and finite-dimensional subspaces of $W$. A finite or infinite
sequence $(x_0,x_1,x_2,x_3,\ldots)$ of vectors is said to be a {\em block sequence}
if for all $n$, $x_n<x_{n+1}$.

Note that, by elementary linear algebra, for all infinite-dimensional
subspaces $X\subseteq W$ there is a subspace $Y\subseteq X$ spanned by an
infinite block sequence, called a {\em block subspace}. Henceforth, we use
variables $X,Y,Z,V$ to denote infinite-dimensional block subspaces of $W$.
Also, denote finite sequences of non-zero vectors
by variables $\vec x, \vec y, \vec z,\vec v$. Finally, variables $E,F$ are used to denote finite-dimensional subspaces of $W$.

\section{Proof of Theorem \ref{tightness}}

We should first recall a natural strengthening of tightness from \cite{minimal}. Suppose $\ku W$ is a Banach space with a Schauder basis $(e_n)$ and find $\go F$ and $W$ as in section \ref{discrete}.
Let also $bb(e_n)\subseteq W^\N$ be the closed set of all block sequences in $W^\N$.   Let $\I$ be the countable set of all non-empty finite intervals $\{n,n+1,\ldots, m\}\subseteq \N$ and give $\I^\N$ the product topology, where $\I$ is taken discrete. We say that $\ku W=[e_n]$ is {\em continuously tight} if there is a continuous function
$$
f\colon bb(e_n)\til \I^\N
$$
such that for any block sequence $(y_n)\in W^\N$, $f\big((y_n)\big)=(I_n)\in \I^\N$ is a sequence of intervals such that $I_0<I_1<I_2<\ldots$ and such  that whenever $A\subseteq \N$ is infinite,
$$
[y_n]\not\sqsubseteq [e_n\del n\notin \bigcup_{k\in A}I_k].
$$
In other words, $f$ continuously chooses the sequence of intervals witnessing tightness.

As in the case of Banach spaces,  for any $K\geqslant 1$, block subspace $Y\subseteq W$, and block sequence $(x_n)$ of $(e_n)$, we define $T((x_n),Y,K)$ to be the non-empty tree consisting of all finite sequences $(y_0,\ldots,y_k)$ in $Y$ such that
$$
(y_0,\ldots,y_k)\sim_K(x_0,\ldots,x_k).
$$
Similarly define the embeddability index of $(x_n)$ in $Y$ by
$$
{\rm Emb}((x_n),Y)=\sup_{K\geqslant 1}{\rm rank}\big(T((x_n),Y,K)\big).
$$
Then, if $\ku Y$ denotes the closed $\R$-linear subspace of $\ku W$ spanned by $Y$, we have, as was observed earlier, that 
$$
{\rm Emb}((x_n),\ku Y)={\rm Emb}((x_n),Y).
$$

We recall the statement of Theorem \ref{tightness}.
\begin{thm}
Let $\ku W$ be a Banach space with a Schauder basis $(e_n)$ and having no minimal subspaces. Then there is a block subspace $\ku X=[x_n]$ that is  $\alpha$-tight for some countable ordinal $\alpha$.
\end{thm}

\begin{proof}
By the results of \cite{minimal}, we have that, as $\ku W$ has no minimal subspaces, there is a block subspace $X=[x_n]$ of $W=[e_n]$ that is continuously tight as witnessed by a function $f$. So it suffices to show that for some $\alpha<\om _1$ and any block sequence $(y_n)$ of $(x_n)$, if $(I_n)=f\big((y_n)\big)$, then 
$$
{\rm Emb}\big((y_n),[x_n\del n\notin \bigcup_{k\in A}I_k]\big )\leqslant\alpha,
$$
for any infinite set $A\subseteq \N$.

Note that if $D$ is any countable set, we can equip the power set $\ku P(D)$ with the compact metric topology obtained from the natural identification with $2^D$. 
Let $[\N]$ denote the space of infinite subsets of $\N$ equipped with the Polish topology induced from $\ku P(\N)$. We define a Borel measurable function between Polish spaces
$$
T\colon bb(x_n)\times [\N]\times \N\til \ku P(X^{<\N}),
$$
by setting 
$$
T((y_n),A,K)=T((y_n),[x_n\del n\notin \bigcup_{j\in A}I_j],K),
$$
where $(I_n)=f\big((y_n)\big)$.

By assumption, the image of $T$ is an analytic set of well-founded trees on $X$. So, by the Boundedness Theorem for analytic sets of  well-founded trees, there is some $\alpha<\om_1$ such that
$$
\sup_{((y_n),A,K)\in bb(x_n)\times [\N]\times \N} {\rm rank}\big(T((y_n),[x_n\del n\notin \bigcup_{j\in A}I_j],K)\big)\leqslant\alpha,
$$
whereby, for any block sequence $(y_n)$ of $(x_n)$ and any infinite subset $A\subseteq \N$,
$$
{\rm Emb}\big((y_n),[x_n\del n\notin \bigcup_{k\in A}I_k]\big )\leqslant\alpha,
$$
showing that $\ku X$ is $\alpha$-tight. 
\end{proof}

\section{Proof of Theorem \ref{main5} }

\subsection{Generalised  $\alpha$-games}
Suppose $X\subseteq W$ and $\alpha$ is a countable ordinal number. We define  the {\em generalised Gowers $\alpha$-game   below $X$}, denoted $G^\alpha_X$,
between two players I and II as follows: 
$$
{\footnotesize 
\begin{array}{cccccccccccc}
{\bf I}  & Y_0   &				  & Y_1           &  			&&   &        & Y_k                       &                            \\
&\xi_0<\alpha& 			  &\xi_1<\xi_0&  			&&&           & \xi_k<\xi_{k-1} &                                \\
&                      &				  &                   &  		        	&& \ldots&&                                &                                 \\
{\bf II}  &          & F_0\subseteq Y_0 &               & F_1\subseteq Y_1&&&       &                                &  F_k\subseteq Y_k        \\
            &          & x_0\in F_0      &                        & x_1\in F_0+F_1    &     &&   &                               & x_k\in F_0+\ldots+F_k
\end{array}
}
$$
Here $\alpha>\xi_0>\xi_1>\ldots>\xi_k=0$ is a strictly decreasing sequence of ordinals, $Y_l\subseteq X$ are block subspaces, the $F_l\subseteq Y_l$ are finite-dimensional
subspaces, and $x_l\in F_0+F_1+\ldots+F_l$ non-zero vectors. Since I plays a strictly decreasing sequence of ordinals, the game will end once $\xi_k=0$ has been chosen and II has responded with some $x_k$. We then say that the sequence $(x_0,\ldots,x_k)$ of non-zero vectors is the {\em outcome} of the game.

Similarly, we can define the {\em asymptotic $\alpha$-game below $X$}, $F_X^\alpha$,  as follows
$$
{\footnotesize 
\begin{array}{cccccccccccc}
{\bf I}  & n_0   &				  & n_1           &  			&&   &        & n_k                       &                            \\
&\xi_0<\alpha& 			  &\xi_1<\xi_0&  			&&&           & \xi_k<\xi_{k-1} &                                \\
&                      &				  &                   &  		        	&& \ldots&&                                &                                 \\
{\bf II}  &          & n_0<F_0 &               & n_1<F_1 &&&       &                                &  n_k<F_k        \\
            &          & x_0\in F_0      &                        & x_1\in F_0+F_1    &     &&   &                               & x_k\in F_0+\ldots+F_k
\end{array}
}
$$
Here again, $\alpha>\xi_0>\xi_1>\ldots>\xi_k=0$ is a strictly decreasing sequence of ordinals, $n_l$ natural numbers, the $F_l$ are finite-dimensional
subspaces of $[e_i]_{i=n_l+1}^\infty$, and $x_l\in F_0+F_1+\ldots+F_l$ non-zero vectors. The game ends once I has played $\xi_k=0$  and II has responded with some $x_k$. The {\em outcome} is the sequence of non-zero vectors $(x_0,\ldots,x_k)$.

If  $\vec x$ is a finite sequence of non-zero vectors, we define the games $G^\alpha_X(\vec x)$, $F_X^\alpha(\vec x)$ as above, except that the outcome is now $\vec x\con (z_0,\ldots,z_k)$.

We also define adversarial $\alpha$-games by mixing the games above. For this, suppose $E,F$ are finite-dimensional subspaces of $W$ and $\vec z$ is an {\em even-length} sequence of non-zero vectors.

We define $A_X^\alpha(\vec z,E,F)$ by 
$$
{\footnotesize 
\begin{array}{ccccccccccc}
           && n_0<E_0    &                                     &n_1<E_1   &     &     &n_k<E_k               &                              \\
           && x_0        &                                           &x_1             &      &   & x_k&                     \\                                                                             
{\bf I}  && Y_0             &				  & Y_1           &     &      & Y_k                       &                            \\
           &&\xi_0& 			            &\xi_1&                        &  & \xi_k &                                \\
    &&                             &				  &                 &  & \ldots&                                &                                 \\
     & n_0&                   &             n_1                 &                   & n_2          &                                &                                 \\
{\bf II}  &&                    & F_0\subseteq Y_0   &                   &     F_1\subseteq Y_1   &   &                                     &  F_k\subseteq Y_k        \\
            &&                    & y_0                 &                  &                    y_1          &         &           & y_k
\end{array}
}
$$
and $B_X^\alpha(\vec z,E,F)$ by:
$$
{\footnotesize 
\begin{array}{ccccccccccc}
           && E_0\subseteq Y_0    &                                   &E_1\subseteq Y_1   &      &   & &E_k\subseteq Y_k               &                              \\
           && x_0&                                     &x_1&         & && x_k&                     \\                                                                             
{\bf I}  && n_0             &				  & n_1   &        &    &       & n_k                       &                            \\
           &&\xi_0& 			            &\xi_1& &       &   & \xi_k &                                \\
    &&                             &				  &          &        & & \ldots&                                &                                 \\
     & Y_0&                   &             Y_1                 &            & &Y_2               &                         &       &                                 \\
{\bf II}  &&                    &n_0< F_0   &                   &          &n_1<F_1 &                              &       &  n_k<F_k        \\
            &&                    & y_0                 &                  &         &  y_1 &                                  &    & y_k
\end{array}
}
$$
where 
$$
\alpha>\xi_0>\xi_1>\ldots>\xi_k=0
$$
is a decreasing sequence of ordinals, $Y_l\subseteq X$ are block subspaces, and  $n_l$ natural numbers. Moreover, in $A_X^\alpha(\vec z,E,F)$, 
$$
E_l\subseteq X\cap [e_i]_{i=n_l+1}^\infty\qquad\text{and} \qquad F_l\subseteq Y_l
$$
are finite-dimensional subspaces, while in $B_X^\alpha(\vec z,E,F)$,
$$
F_l\subseteq X\cap [e_i]_{i=n_l+1}^\infty\qquad\text{and} \qquad E_l\subseteq Y_l
$$
are finite-dimensional subspaces. Finally, the non-zero vectors $x_l$ and $y_l$ are chosen such that
$$
x_l\in E+E_0+\ldots+E_l,
$$
while 
$$
y_l\in F+F_0+\ldots+F_l.
$$
Both games terminate once I has played $\xi_k=0$ and II has responded with some $y_k$. The {\em outcome} is then the finite sequence of non-zero vectors 
$$
\vec z\con(x_0,y_0,x_1,y_1,\ldots,x_k,y_k).
$$

Now suppose instead that $\vec z$ is an {\em odd-length} sequence of non-zero vectors. 
We then define $A_X^\alpha(\vec z,E,F)$ by 
$$
{\footnotesize 
\begin{array}{ccccccccccc}
           &   &                                     &n_1<E_1   &     &  n_2<E_2 &  &n_k<E_k               &                              \\
           &     &                                           &x_1             &   & x_2  &   & x_k&                     \\                                                                             
{\bf I}  & Y_0             &				  & Y_1           &   &Y_2  &      & Y_k                       &                            \\
           && 			            &\xi_1&                        &     \xi_2 & & \xi_k &                                \\
    &                             &				  &                 &&  & \ldots&                                &                                 \\
   &                   &             n_1                 &                   & n_2          &                                &                      &           \\
{\bf II}  &                    & F_0\subseteq Y_0   &                   &     F_1\subseteq Y_1   &   &                   &                  &  F_k\subseteq Y_k        \\
            &                    & y_0                 &                  &                    y_1          &         &          & & y_k
\end{array}
}
$$
and $B_X^\alpha(\vec z,E,F)$ by:
$$
{\footnotesize 
\begin{array}{ccccccccccc}
           &   &                                   &E_1\subseteq Y_1   &      &E_2\subseteq Y_2  &&  &E_k\subseteq Y_k               &                              \\
           &&                                     &x_1&         &x_2 &&& x_k&                     \\                                                                             
{\bf I}  & n_0             &				  & n_1   &&n_2        &          & & n_k                       &                            \\
           && 			            &\xi_1& &     \xi_2    & & &\xi_k                                 \\
    &                             &				  &          &        & && \ldots&                                &                                 \\
     &                   &             Y_1                 &             &Y_2               &                                &                                 \\
{\bf II}  &                    &n_0< F_0   &                             &n_1<F_1 &      &        &            &           &  n_k<F_k        \\
            &                    & y_0                 &                          &  y_1 &             &        &          &       & y_k
\end{array}
}
$$
where 
$$
\alpha>\xi_1>\ldots>\xi_k=0
$$
is a decreasing sequence of ordinals, 
$$
x_l\in E+E_1+\ldots+E_l,
$$
$$
y_l\in F+F_0+\ldots+F_l,
$$
and otherwise the games are identical to those above. The {\em outcome} is now the finite sequence
$\vec z\con (y_0,x_1,y_1,\ldots,x_k,y_k)$.

If $\vec z=\tom$ and $E=F=\{0\}$, we shall write $A^\alpha_X$ and $B^\alpha_X$ instead of $A_X^\alpha(\vec z,E,F)$, respectively $B_X^\alpha(\vec z,E,F)$.
Thus, in both games $A^\alpha_X$ and $B^\alpha_X$, one should remember that I is the {\em
first} to play a vector. And in $A^\alpha_X$, I plays block subspaces and II plays
integers, while in $B^\alpha_X$, II takes the role of playing block subspaces
and I plays integers.

We should also mention the degenerate case when $\alpha=0$. The games $G^\alpha_X(\vec z)$ and $F^\alpha_X(\vec z)$ then terminate immediately with outcome $\vec z$ and, if $\vec z$ is of even length, the same holds for the games $A^\alpha_X(\vec z,E,F)$ and $B^\alpha_X(\vec z,E,F)$. On the other hand, if $\vec z$ is of odd length, in $A^\alpha_X(\vec z,E,F)$ and $B^\alpha_X(\vec z,E,F)$, I will play respectively $Y_0$ and $n_0$ and II respond with a single $y_0$ according to the rules, whereby the outcome is now $\vec z\con y_0$.

If $X$ and $Y$ are subspaces, where $Y$ is spanned by an infinite block
sequence $(y_0,y_1,y_2,\ldots)$, we write $Y\subseteq^* X$ if there is
$n$ such that $y_m\in X$ for all $m\geqslant n$. A simple diagonalisation argument
shows that if $X_0\supseteq X_1\supseteq X_2\supseteq \ldots$ is a decreasing
sequence of block subspaces, then there is some $Y\subseteq X_0$ such that
$Y\subseteq^* X_n$ for all $n$.

The aim of the games above is for each of the players to ensure that the
outcome lies in some predetermined set depending on the player. By
the asymptotic nature of the game, it is easily seen that if $T\subseteq
W^{<\N}$ and $Y\subseteq^* X$, then if II has a strategy in $G^\alpha_X$ or $A^\alpha_X(\vec z,E,F)$ to play in
$T$, i.e., to ensure that the outcome is in $T$, then II will have a strategy
in $G^\alpha_Y$, respectively $A^\alpha_Y(\vec z,E,F)$, to play in $T$ too. Similarly, if I has a strategy in $F^\alpha_X$ or $B^\alpha_X(\vec z,E,F)$ to play
in $T$, then I also has a strategy in $F^\alpha_Y$, respectively in $B^\alpha_X(\vec z,E,F)$, to play in $T$.

\subsection{Ramsey determinacy of adversarial $\alpha$-games}
We are now ready to prove the basic determinacy theorem for adversarial $\alpha$-games, which can be seen as a refinement of the determinacy theorem for open adversarial games (see Theorem 12 in \cite{exact}).

\begin{thm}\label{relational}
Suppose $\alpha<\om_1$ and $T\subseteq W^{<\N}$. Then for any $X\subseteq W$ there is $Y\subseteq X$
such that either
\begin{enumerate}
  \item II has a strategy in $A^\alpha_Y$ to play in $T$, or
  \item I has a strategy in $B^\alpha_Y$ to play in $\sim\!T$.
\end{enumerate}
\end{thm}

\begin{proof}
We say that
\begin{itemize}
  \item [(a)] $(\vec x,E,F,\beta,X)$ is {\em good} if II has a strategy in $A^\beta_X(\vec x,E,F)$ to play in
  $T$.
  \item [(b)] $(\vec x,E,F,\beta,X)$ is {\em bad} if $\a  Y\subseteq X$, $(\vec x,E,F,\beta,Y)$ is not good.
  \item [(c)] $(\vec x,E,F,\beta,X)$ is {\em worse} if it is bad and either
  \begin{enumerate}
    \item $|\vec x|$ is even and $\beta=0$, or
    \item $|\vec x|$ is even, $\beta>0$, and
    $$
    \a Y\subseteq X\;\e E_0\subseteq Y\; \e x_0\in E+E_0\;\e \gamma<\beta\;(\vec x\con x_0,E+E_0,F,\gamma,X) \textrm{ is bad},
    $$
    or 
    \item $|\vec x|$ is odd and
    $$
    \e n\;\a n<F_0\subseteq X\;\a y_0\in F+F_0\;(\vec x\con y_0,E,F+F_0,\beta,X) \textrm{ is
    bad},
    $$ 
  \end{enumerate}
  \item[(d)] $(\vec x,E,F,\beta,X)$ is {\em wicked} if $\a y_0\in F\;(\vec x\con y_0,E,F,\beta,X)$ is bad.  
\end{itemize}
One checks that good, bad and wicked are all $\subseteq^*$-hereditary in the last coordinate, that is, if $(\vec x,E,F,\beta,X)$ is good and $Y\subseteq^*X$, then also $(\vec x,E,F,\beta,Y)$ is good, etc. So, by diagonalising over the countably many tuples of  $\vec x$, $E$, $F$, and $\beta\leqslant \alpha$, we can find some $Y\subseteq X$ such that for all $\vec x$, $E$, $F$, and $\beta\leqslant \alpha$, 
\begin{itemize}
\item[(i)] $(\vec x, E, F,\beta,Y)$ is either good or bad, and
\item[(ii)] if there is some $Y_0\subseteq Y$ such that for all $F_0\subseteq Y_0$,  $(\vec x, E, F+F_0,\beta,Y)$ is wicked, then there is some $n$ such that for all $n<F_0\subseteq Y$,  $(\vec x, E, F+F_0,\beta,Y)$ is wicked.
\end{itemize}

\begin{lemme}
If $(\vec x,E,F,\beta,Y)$ is bad, then it is worse.
\end{lemme}

\begin{proof}
Assume first that $|\vec x|$ is even. The case when $\beta=0$ is trivial, so assume also $\beta>0$. Since $(\vec x,E,F,\beta,Y)$ is bad, we have 
\begin{displaymath}\begin{split}
\a V\subseteq Y \text{ II has no strategy in $A^\beta_V(\vec x,E,F)$ to play in }T. 
\end{split}\end{displaymath}
Referring to the definition of the game $A^\beta_V(\vec x,E,F)$, this implies that
\begin{displaymath}\begin{split}
\a V\subseteq Y\;& \e E_0\subseteq V\; \e x_0\in E+E_0\; \e \gamma<\beta\; \\
&\text{ II has no strategy in $A^\gamma_V(\vec x\con x_0,E+E_0,F)$ to play in }T,
\end{split}\end{displaymath}
(note that the subspace $Y_0\subseteq V$ also played by I becomes the first play of I in the game $A^\gamma_V(\vec x\con x_0,E+E_0,F)$).
But if $V\subseteq Y$ and II has no strategy in $A^\gamma_V(\vec x\con x_0,E+E_0,F)$ to play in $T$, then $(\vec x\con x_0,E+E_0,F,\gamma,V)$ is not good and hence must be bad. Thus, 
\begin{displaymath}\begin{split}
\a V\subseteq Y\;\e E_0\subseteq V\; &\e x_0\in E+E_0\; \e \gamma<\beta\; (\vec x\con x_0,E+E_0,F,\gamma,V) \text{ is bad,}
\end{split}\end{displaymath}
which is just to say that $(\vec x,E,F,\beta,Y)$ is worse.

Now suppose instead that $|\vec x|$ is odd. As $(\vec x,E,F,\beta,Y)$ is bad, it is not good and so II has no strategy in $A^\beta_{Y}(\vec x, E,F)$ to play in $T$. Therefore, for some $Y_0\subseteq Y$, we have 
$$
\a F_0\subseteq Y_0\; \a y_0\in F+F_0 \text{ II has no strategy in $A^\beta_Y(\vec x\con y_0,E,F+F_0)$ to play in }T.
$$
i.e., 
$$
\a F_0\subseteq Y_0\; \a y_0\in F+F_0 \;(\vec x\con y_0,E,F+F_0, \beta,Y)\text{ is not good and hence is bad.}
$$
In other words, 
$$
\a F_0\subseteq Y_0\; (\vec x,E,F+F_0,\beta,Y) \text{ is wicked.}
$$
So by (ii) we have 
$$
\e n\; \a n<F_0\subseteq Y\;(\vec x,E,F+F_0,\beta,Y) \text{ is wicked},
$$
that is
$$
\e n\; \a n<F_0\subseteq Y\;\a y_0\in F+F_0\; (\vec x\con y_0,E,F+F_0,\beta,Y) \text{ is bad,}
$$
showing that $(\vec x,E,F,\beta,Y)$ is worse.
\end{proof}
If $(\tom, \{0\},\{0\},\alpha,Y)$ is good, the first possibility of the statement of the theorem holds. So suppose instead  $(\tom, \{0\},\{0\},\alpha,Y)$ is bad and hence worse. Then, using the lemma and unraveling the definition of worse, we see that I has a strategy to play the game $B^\alpha_Y$ such that at any point in the game, if 
\begin{align*}
&\vec x=(x_0,y_0,x_1,y_1,\ldots,x_l,y_l)\\
&E_0,F_0,E_1,F_1,\ldots,E_l,F_l\\
&\alpha>\xi_0>\xi_1>\ldots>\xi_l,
\end{align*}
respectively,
\begin{align*}
&\vec y=(x_0,y_0,x_1,y_1,\ldots,y_{l-1},x_l)\\
&E_0,F_0,E_1,F_1,\ldots,F_{l-1},E_l\\
&\alpha>\xi_0>\xi_1>\ldots>\xi_l,
\end{align*}
have been played, then 
$$
(\vec x,E_0+\ldots+E_l,F_0+\ldots+F_l,\xi_l,Y),
$$ 
respectively 
$$
(\vec y,E_0+\ldots+E_l,F_0+\ldots+F_{l-1},\xi_l,Y),
$$ 
is worse.
Since $\alpha>\xi_0>\xi_1\ldots$, we eventually have $\xi_k=0$, that is, the game terminates with some worse 
$$
(\vec z,E_0+\ldots+E_k,F_0+\ldots+F_k,0,Y),
$$
whereby the outcome $\vec z$ lies in $\sim\!T$.
\end{proof}

\subsection{A game theoretic dichotomy}
We first need a lemma ensuring us a certain uniformity.
\begin{lemme}\label{uniformity}
Let $\beta<\om_1$ and suppose that for every $X\subseteq W$ there are $K\geqslant 1$ and a block sequence $(y_n)\subseteq X$ such that II has a strategy in $F^\beta_X$ to play  $(x_0,x_1,\ldots,x_k)$ satisfying
$$
(x_0,x_1,\ldots,x_k)\sim_K(y_0,y_1,\ldots,y_k).
$$
Then there are  $K\geqslant 1$ and $Y\subseteq W$ such that for all $X\subseteq Y$ there is a block sequence $(y_n)\subseteq X$ such that II has a strategy in $F^\beta_X$ to play  $(x_0,x_1,\ldots,x_k)$ satisfying
$$
(x_0,x_1,\ldots,x_k)\sim_K(y_0,y_1,\ldots,y_k).
$$
In other words, $K\geqslant 1$ can be chosen uniformly for all $X\subseteq Y$.
\end{lemme}

\begin{proof}
Assume toward a contradiction that the conclusion fails. Then, as the games $F^\beta_X$ to play in any set $T\subseteq W^{<\N}$ are determined, i.e., either I or II has a winning strategy, we can inductively define $W\supseteq Y_0\supseteq Y_1\supseteq \ldots$ such that for any 
block sequence $(y_n)$ in $Y_K$, I has a strategy in $F^\beta_{Y_K}$ 
to play  $(x_0,x_1,\ldots,x_k)$ satisfying
$$
(x_0,x_1,\ldots,x_k)\not\sim_K(y_0,y_1,\ldots,y_k).
$$
For each $N\in \N$, let $c(N)$ be a constant such that if $(v_0,v_1,\ldots,v_{N-1},v_{N},v_{N+1},\ldots)$ and $(u_0,u_1,\ldots,u_{N-1},v_{N},v_{N+1},\ldots)$ are two normalised block sequences of $(e_n)$, then 
$$
(v_0,v_1,\ldots,v_{N-1},v_{N},v_{N+1},\ldots)\sim_{c(N)}(u_0,u_1,\ldots,u_{N-1},v_{N},v_{N+1},\ldots).
$$
Now choose a block sequence $(x_0,x_1,x_2,\ldots)$ such that for every $N$ there are normalised $v_0,v_1,\ldots,v_{N-1}\in Y_{N\cdot c(N)}$ with
$$
v_0<v_1<\ldots<v_{N-1}<x_N<x_{N+1}<\ldots
$$
and, moreover, such that $x_N,x_{N+1},\ldots\in Y_{N\cdot c(N)}$.  Set also $X=[x_n]$.

By the assumptions of the lemma, we can find some  constant $N\in \N$ and a normalised block sequence $(y_0,y_1,\ldots)$ in $X$ such that II has a strategy in $F^\beta_X$ to play  $(w_0,w_1,\ldots,w_k)$ with 
$$
(w_0,w_1,\ldots,w_k)\sim_N(y_0,y_1,\ldots,y_k).
$$
Since $\min {\rm supp}(x_N)\leqslant\min{\rm supp}(y_N)$, it follows by the choice of $(x_n)$ that there are normalised $v_0,v_1,\ldots,v_{N-1}\in Y_{N\cdot c(N)}$ such that
$$
v_0<v_1<\ldots<v_{N-1}<y_N<y_{N+1}<\ldots.
$$
Moreover, by the definition of $c(N)$, we have
$$
(v_0,v_1,\ldots,v_{N-1},y_{N},y_{N+1},\ldots)\sim_{c(N)}(y_0,y_1,\ldots,y_{N-1},y_{N},y_{N+1},\ldots).
$$
Thus, if we let $v_n=y_n$ for all $n\geqslant N$, we see that II has a strategy in $F_X^\beta$ to play 
$(w_0,w_1,\ldots,w_k)$ with 
$$
(w_0,w_1,\ldots,w_k)\sim_N(y_0,y_1,\ldots,y_k)\sim_{c(N)}(v_0,v_1,\ldots,v_k).
$$
But $X\subseteq^*Y_{N\cdot c(N)}$, so II has a strategy in $F^\beta_{Y_{N\cdot c(N)}}$ to play 
$(w_0,w_1,\ldots,w_k)$ with 
$$
(w_0,w_1,\ldots,w_k)\sim_{N\cdot c(N)}(v_0,v_1,\ldots,v_k).
$$
On the other hand, $(v_n)\subseteq Y_{N\cdot c(N)}$ and so I has a strategy in $F^\beta_{Y_{N\cdot c(N)}}$ 
to play  $(w_0,w_1,\ldots,w_k)$ such that
$$
(w_0,w_1,\ldots,w_k)\not\sim_{N\cdot c(N)}(v_0,v_1,\ldots,v_k),
$$
which is absurd. This contradiction proves the lemma.
\end{proof}

\begin{lemme}\label{going to adversarial}
Suppose $X\subseteq W$, $(y_0,y_1,y_2,\ldots)$ is a sequence of vectors in $W$, $\alpha<\om_1$ and $K\geqslant 1$. Assume that II has a strategy in $F_X^{\om\cdot \alpha}$ to play $(x_0,x_1,\ldots,x_k)$ such that
$$
(x_0,x_1,\ldots,x_k)\sim_K(y_0,y_1,\ldots,y_k).
$$
Then II has a strategy in $B^{\alpha}_X$ to play $(u_0,v_0,u_1,v_1,\ldots,u_k,v_k)$ such that 
$$
(u_0,u_1,\ldots,u_k)\sim_K(v_0,v_1,\ldots,v_k).
$$
\end{lemme}

\begin{proof}We shall describe the strategy for II in the game $B^\alpha_X$, the idea being that, when playing the game $B_X^\alpha$, II will keep track of an auxiliary run of $F^{\om \cdot\alpha}_X$, using his strategy there to compute his moves in $B^\alpha_X$.

Now, in $B^\alpha_X$, II will play subspaces $Y_0,Y_1, \ldots$ all equal to $Y=[y_n]$, whereby the subspaces $Y_0,Y_1,\ldots$ and $E_0,E_1,\ldots$ lose their relevance and we can eliminate them from the game for simplicity of notation. We thus have the following presentation of the game $B_X^\alpha$.
$$
{\footnotesize 
\begin{array}{ccccccccccc}
           && u_0\in Y&                                     &u_1\in Y&        &  & u_k\in Y&                     \\                                                                             
{\bf I}  && n_0             &				  & n_1           &     &      & n_k                       &                            \\
           &&\xi_0<\alpha& 			            &\xi_1<\xi_0&        &   & \xi_k<\xi_{k-1}      &                                \\
    &&                             &				  &               &    & \ldots&                                &                                 \\
{\bf II}  &&                    &n_0< F_0   &                   & n_1<F_1         & &                                     &  n_k<F_k        \\
            &&                    & v_0\in F_0   &                  & v_1\in F_0+F_1    &       &                                      & v_k\in F_0+\ldots+F_k
\end{array}
}
$$
So suppose $u_0,u_1,\ldots$ is being played by I in $B^\alpha_X$. To compute the answer $v_0,v_1,\ldots$, II follows his strategy in $F^{\om\cdot\alpha}_X$ to play $(z_0,z_1,\ldots,z_k)\sim_K(y_0,y_1,\ldots,y_k)$ as follows. First, as $u_0,u_1,\ldots\in Y=[y_n]$, we can write each $u_i$ as
$$
u_i=\sum_{j=0}^{m_i-1}\lambda^i_jy_j,
$$
where we, by adding dummy variables, can assume that $m_0<m_1<m_2<\ldots$.
So to compute $v_0$ and $F_0$ given $u_0$, $n_0$ and $\xi_0$, II first runs an initial part of $F^{\om\cdot\alpha}_X$ as follows
$$
{\footnotesize 
\begin{array}{cccccccccccc}
{\bf I}  & n_0   &				  & n_0           &  			&        & n_0                       &                            \\
&\om\xi_0\!+\!m_0\!-\!1& 			  &\om\xi_0\!+\!m_0\!-\!2&  			&           & \om\xi_0&                                \\
&                      &				  &                   &  		        	& \ldots&                                &                                 \\
{\bf II}  &          & n_0<F^0_1 &               & n_0<F^0_2&&       &                                 n_0<F^0_{m_0}        \\
            &          & x_0\in F^0_1      &                        & x_1\in F^0_1\!+\!F^0_2    &        &                               & x_{m_0\!-\!1}\in F^0_1\!+\ldots+\!F^0_{m_0}
\end{array}
}
$$
He then plays $F_0=F^0_1+\ldots+F^0_{m_0}$ and 
$$
v_0=\sum_{j=0}^{m_0-1}\lambda^0_jx_j\in F_0
$$
in $B^\alpha_X$.

Next, I will play some $u_1$, $n_1$ and $\xi_1$, and, to compute $v_1$ and $F_1$, II will continue the above run of $F^{\om\cdot\alpha}_X$ with
$$
{\footnotesize 
\begin{array}{cccccccccccc}
{\bf I}  & n_1   &				        &   			&        & n_1                       &                            \\
&\om\xi_1+m_1-1& 			 &  			&           & \om\xi_1&                                \\
&                      &				                     &  		        	& \ldots&                                &                                 \\
{\bf II}  &          & n_1<F^1_1 &               &&       &                                 n_1<F^1_{m_1}        \\
            &          & x_{m_0}\in F_0+F^0_1      &                        &        &                               & x_{m_1-1}\in F_0+ F^1_1+\ldots+F^1_{m_1}
\end{array}
}
$$
He then plays $F_1=F^1_1+\ldots+F^{1}_{m_1}$ and 
$$
v_1=\sum_{j=0}^{m_1-1}\lambda^1_jx_j\in F_0+F_1
$$
in $B^\alpha_X$.

So at each stage, II will continue his run of $F^{\om\cdot\alpha}_X$ a bit further until eventually I has played some $\xi_k=0$.
Thus, in the game $F_X^{\om\cdot\alpha}$, I will play ordinals
$$
\alpha>\om\xi_0+m_0-1>\om\xi_0+m_0-2>\ldots>\om\xi_0>\om\xi_1+m_1-1>\ldots>\om\xi_k=0
$$
and integers $n_0\geqslant n_0\geqslant\ldots\geqslant n_0\geqslant n_1\geqslant \ldots\geqslant n_k$, while II will use his strategy to play $(x_0,x_1,\ldots,x_{m_k-1})$ such that
$$
(x_0,x_1,\ldots,x_{m_k-1})\sim_K(y_0,y_1,\ldots,y_{m_k-1}).
$$
Since the $v_i$ and $u_i$ have the same coefficients over respectively $(x_n)$ and $(y_n)$, it follows that 
$$
(u_0,u_1,\ldots,u_k)\sim_K(v_0,v_1,\ldots,v_k).
$$
\end{proof}

By a  similar argument, we have the following lemma.
\begin{lemme}\label{passing to block sequences}
Suppose $X\subseteq W$, $(y_0,y_1,y_2,\ldots)$ is a block sequence in $W$, $\alpha<\om_1$ and $K\geqslant 1$. Assume that II has a strategy in $F_X^{\om\cdot \alpha}$ to play $(x_0,x_1,\ldots,x_k)$ such that
$$
(x_0,x_1,\ldots,x_k)\sim_K(y_0,y_1,\ldots,y_k).
$$
Then for any block sequence $(z_n)$ in $[y_n]$, II has a strategy in $F^{\alpha}_X$ to play $(v_0,v_1,\ldots,v_k)$ such that
$$
(v_0,v_1,\ldots,v_k)\sim_K(z_0,z_1,\ldots,z_k).
$$
\end{lemme}

\begin{proof}
First, as $(z_n)$ is a block sequence in $[y_n]$, we can write each $z_i$ as
$$
z_i=\sum_{j=m_{i-1}}^{m_i-1}\lambda_jy_j,
$$
where $m_{-1}=0<m_0<m_1<m_2<\ldots$.

As before, when playing $F_X^\alpha$, II will keep track of an auxiliary run of $F^{\om\alpha}_X$, using his strategy there to compute his moves in $F^\alpha_X$. So the game  $F^\alpha_X$ runs as follows:
$$
{\footnotesize 
\begin{array}{ccccccccccc}
                                                               
{\bf I}  && n_0             &				  & n_1           &     &      & n_k                       &                            \\
           &&\xi_0& 			            &\xi_1&        &   & \xi_k   &                                \\
    &&                             &				  &               &    & \ldots&                                &                                 \\
{\bf II}  &&                    &n_0< F_0   &                   & n_1<F_1         & &                                     &  n_k<F_k        \\
            &&                    & v_0\in F_0   &                  & v_1\in F_0+F_1    &       &                                      & v_k\in F_0+\ldots+F_k
\end{array}
}
$$
To compute $v_0$,  II first runs an initial part of $F^{\om\alpha}_X$ as follows
$$
{\footnotesize 
\begin{array}{cccccccccccc}
{\bf I}  & n_0   &				  & n_0           &  			&        & n_0                       &                            \\
&\om\xi_0\!+\!m_0\!-\!1& 			  &\om\xi_0\!+\!m_0\!-\!2&  			&           & \om\xi_0&                                \\
&                      &				  &                   &  		        	& \ldots&                                &                                 \\
{\bf II}  &          & n_0<F^0_1 &               & n_0<F^0_2&&       &                                 n_0<F^0_{m_0}        \\
            &          & x_0\in F^0_1      &                        & x_1\in F^0_1\!+\!F^0_2    &        &                               & x_{m_0\!-\!1}\in F^0_1\!+\ldots+\!F^0_{m_0}
\end{array}
}
$$
He then plays $F_0=F^0_1+\ldots+F^0_{m_0}$ and 
$$
v_0=\sum_{j=m_{-1}}^{m_0-1}\lambda_jx_j\in F_0
$$
in $F^\alpha_X$.

Next, I will play some $\xi_1$ and $n_1$ and to compute $v_1$ and $F_1$, II will continue the above run of $F^{\om\alpha}_X$ with
$$
{\footnotesize 
\begin{array}{cccccccccccc}
{\bf I}  & n_1   &				        &   			&        & n_1                       &                            \\
&\om\xi_1+m_1-m_0-1& 			 &  			&           & \om\xi_1&                                \\
&                      &				                     &  		        	& \ldots&                                &                                 \\
{\bf II}  &          & n_1<F^1_1 &               &&       &                                 n_1<F^1_{m_1-m_0}        \\
            &          & x_{m_0}\in F_0+F^0_1      &                        &        &                               & x_{m_1-1}\in F_0+ F^1_1+\ldots+F^1_{m_1-m_0}
\end{array}
}
$$
He then plays $F_1=F^1_1+\ldots+F^{1}_{m_1-m_0}$ and 
$$
v_1=\sum_{j=m_0}^{m_1-1}\lambda_jx_j\in F_0+F_1
$$
in $F^\alpha_X$.

So at each stage, II will continue his run of $F^{\om\alpha}_X$ a bit further until eventually I has played some $\xi_k=0$.
Thus, in the game $F_X^{\om\alpha}$, I will play ordinals
$$
\alpha>\om\xi_0+m_0-1>\om\xi_0+m_0-2>\ldots>\om\xi_0>\om\xi_1+m_1-m_0-1>\ldots>\om\xi_k=0
$$
and integers $n_0\geqslant n_0\geqslant\ldots\geqslant n_0\geqslant n_1\geqslant \ldots\geqslant n_k$, while II will use his strategy to play $(x_0,x_1,\ldots,x_{m_k-1})$ such that
$$
(x_0,x_1,\ldots,x_{m_k-1})\sim_K(y_0,y_1,\ldots,y_{m_k-1}).
$$
Since the $v_i$ and $z_i$ have the same coefficients over respectively $(x_n)$ and $(y_n)$, it follows that 
$$
(v_0,v_1,\ldots,v_k)\sim_K(z_0,z_1,\ldots,z_k).
$$
\end{proof}

\begin{lemme}\label{gowers game}
Suppose $X\subseteq W$, $(y_n)$ is a block sequence in $W$, $\alpha<\om_1$, and $K,C\geqslant 1$.
Assume that 
\begin{itemize}
\item[(a)] II has a strategy in $F^\alpha_X$ to play $(x_0,\ldots,x_k)$ such that
$$
(x_0,x_{1},\ldots,x_{k})\sim_K(y_0,y_{1},\ldots,y_{k}), 
$$
and
\item[(b)] II has a strategy in $A_X^\alpha$ to play $(u_0,v_0,\ldots,u_k,v_k)$ such that
$$
(u_0,u_{1},\ldots,u_{k})\sim_C(v_0,v_{1},\ldots,v_{k}), 
$$
\end{itemize}
Then II has a strategy in $G_X^\alpha$ to play $(v_0,\ldots,v_k)$ such that
$$
(v_0,v_{1},\ldots,v_{k})\sim_{KC}(y_0,y_{1},\ldots,y_{k}). 
$$
\end{lemme}

\begin{proof}
To compute his strategy in $G_X^\alpha$, II will play auxiliary runs of the games $A_X^\alpha$ and $F^\alpha_X$ in which he is using the strategies described above. Information is then copied between the games as indicated in the diagrams below.

The game $G^\alpha_X$:
$$
{\footnotesize 
\begin{array}{cccccccccccc}
{\bf I}   & Y_0      &   	                 & Y_1 	   &                                            &&            &&Y_k                       \\
           &\xi_0     & 		        & \xi_1           &                                            &&            & &\xi_k                                           \\
           &              &                        &                     &                                            &&\ldots  &   &                           \\
{\bf II}  &             &F_0\subseteq Y_0  &          &F_1\subseteq Y_1            &&           & &          & F_k\subseteq Y_k \\
            &             & v_0\in F_0    &                     &v_{1}\in F_0+F_{1}            &&           & &          &v_{k}\in F_0+\ldots+F_k
 \end{array}
}
$$
The game $F_X^\alpha$:
$$
{\footnotesize 
\begin{array}{cccccccccccc}
{\bf I}   & n_0      &   	                 & n_1 	   &                                            &&            &&n_k                       \\
           &\xi_0     & 		        & \xi_1           &                                            &&            & &\xi_k                                           \\
           &              &                        &                     &                                            &&\ldots  &   &                           \\
{\bf II}  &             &n_0<E_0         &                    &n_1<E_1                             &&           & &          & n_k<E_k                  \\
            &             & x_0\in E_0    &                     &x_{1}\in E_0+E_{1}            &&           & &          &x_{k}\in E_0+\ldots+E_k
 \end{array}
}
$$
The game $A_X^\alpha$:
$$
{\footnotesize 
\begin{array}{ccccccccccc}
           && n_0<E_0    &                                     &n_1<E_1   &          &n_k<E_k               &                              \\
           && x_0\in E_0 &                                     &x_1\in E_0\!+\!E_1&          & x_k\in E_0\!+\!\ldots\!+\!E_k&                     \\                                                                             
{\bf I}  && Y_0             &				  & Y_1           &           & Y_k                       &                            \\
           &&\xi_0             & 			            &\xi_1          &           & \xi_k                      &                                \\
    &&                             &				  &                   & \ldots&                                &                                 \\
     &n_0&                   &             n_1                 &                   &           &                                &                                 \\
{\bf II}  &&                    & F_0\subseteq Y_0   &                   &           &                                     &  F_k\subseteq Y_k        \\
            &&                    &  v_0\in F_0                &                  &            &                                      & v_k\in F_0\!+\!\ldots\!+\!F_k
\end{array}
}
$$

By chasing the diagrams, one sees that this fully determines how II is to play in $G^\alpha_X$. Moreover, since II follows his strategy in $F_X^\alpha$, we have 
$$
(x_0,x_{1},\ldots,x_{k})\sim_K(y_0,y_{1},\ldots,y_{k}), 
$$
while the strategy in $A_X^\alpha$ ensures that 
$$
(x_0,x_{1},\ldots,x_{k})\sim_C(v_0,v_{1},\ldots,v_{k}),
$$
from which the conclusion follows.
\end{proof}

\begin{thm}\label{main1}
Suppose $\alpha<\om_1$. Then there is  $X\subseteq W$ such that one of the following holds
\begin{enumerate}
\item For every block sequence $(y_n)$ in $X$ and $K\geqslant 1$, I has a strategy in $F^{\om\alpha}_X$ to play $(x_0,x_1,\ldots,x_k)$ satisfying
$$
(x_0,x_1,\ldots,x_k)\not\sim_K(y_0,y_1,\ldots,y_k).
$$
\item For some $K\geqslant 1$ and every block sequence $(z_n)\subseteq X$, II has a strategy in $G^\alpha_X$ to play $(x_0,x_1,\ldots,x_k)$ satisfying
$$
(x_0,x_1,\ldots,x_k)\sim_K(z_0,z_1,\ldots,z_k).
$$
\end{enumerate}
\end{thm}

\begin{proof}
Suppose that there is no $X\subseteq W$ for which (1) holds. Then, using that the game $F_X^{\om\alpha}$ is determined,  for every $X\subseteq W$ there is a block sequence $(y_n)$ in $X$ and some $K\geqslant 1$ such that II has a strategy in $F^{\om\alpha}_X$ to play $(x_0,x_1,\ldots,x_k)$ satisfying
$$
(x_0,x_1,\ldots,x_k)\sim_K(y_0,y_1,\ldots,y_k).
$$
So, by Lemma \ref{uniformity}, there is some $K\geqslant 1$ and $Y\subseteq W$ such that  for all $X\subseteq Y$ there is some block 
sequence $(y_n)$  in $X$ such that II has a strategy in $F^{\om\alpha}_X$ to play $(x_0 , x_1 , \ldots, x_k )$ satisfying
$$
(x_0 , x_1 , \ldots, x_k )\sim_K(y_0 , y_1 , \ldots, y_k ).
$$

If thus follows from Lemma \ref{going to adversarial} that for all $X\subseteq Y$,  II has a strategy in $B^{\alpha}_X$ to play $(u_0,v_0,u_1,v_1,\ldots,u_k,v_k)$ such that 
$$
(u_0,u_1,\ldots,u_k)\sim_K(v_0,v_1,\ldots,v_k).
$$
Therefore, there is no $X\subseteq Y$ such that I has a strategy in $B^{\alpha}_X$ to play a sequence 
$(u_0,v_0,u_1,v_1,\ldots,u_k,v_k)$ satisfying
$$
(u_0,u_1,\ldots,u_k)\not\sim_K(v_0,v_1,\ldots,v_k),
$$
and thus, by Theorem \ref{relational}, we can find some $X\subseteq Y$ such that II has a strategy in $A^\alpha_X$ to play 
$(u_0,v_0,u_1,v_1,\ldots,u_k,v_k)$ satisfying 
$$
(u_0,u_1,\ldots,u_k)\sim_K(v_0,v_1,\ldots,v_k).
$$

Let $(y_n)$ be the block sequence in $X$ such that  II has a strategy in $F^{\om\alpha}_X$ to play $(x_0 , x_1 , \ldots, x_k )$ satisfying 
$$
(x_0 , x_1 , \ldots, x_k )\sim_K(y_0 , y_1 , \ldots, y_k ).
$$
Then, using Lemma \ref{passing to block sequences}, we see that for any block sequence $(z_n)\subseteq [y_n]$, II has a strategy in $F^{\alpha}_X$ to play $(x_0 , x_1 ,\ldots , x_k )$ such 
that 
$$
(x_0 , x_1 , \ldots, x_k )\sim_K(z_0 , z_1 , \ldots , z_k ).
$$

In other words, there is some block sequence $(y_n)$ in $X$ such that for any block sequence $(z_n)\subseteq [y_n]$
\begin{itemize}
\item[(a)] II has a strategy in $F^\alpha_X$ to play $(x_0,\ldots,x_k)$ satisfying
$$
(x_0,x_{1},\ldots,x_{k})\sim_K(z_0,z_{1},\ldots,z_{k}), 
$$
and
\item[(b)] II has a strategy in $A_X^\alpha$ to play $(u_0,v_0,\ldots,u_k,v_k)$ satisfying
$$
(u_0,u_{1},\ldots,u_{k})\sim_K(v_0,v_{1},\ldots,v_{k}), 
$$
\end{itemize}
So finally, by Lemma \ref{gowers game}, for any block sequence $(z_n)\subseteq [y_n]$, II has a strategy in $G_X^\alpha$ to play $(v_0,\ldots,v_k)$ such that
$$
(v_0,v_{1},\ldots,v_{k})\sim_{K^2}(z_0,z_{1},\ldots,z_{k}). 
$$
Replacing $X$ by the block subspace $[y_n]\subseteq X$ and $K$ by $K^2$, we get (2).
\end{proof}

\subsection{The embeddability index}

\begin{lemme}\label{loose}
Suppose $\alpha<\om_1$, $K\geqslant 1$, $X\subseteq W$ and $(z_n)\subseteq W$ is a block sequence such that  II has a strategy in $G^\alpha_X$ to play $(y_0,\ldots,y_k)$ satisfying 
$$
(y_0,\ldots,y_k)\sim_K(z_0,\ldots,z_k).
$$
Then for any subspace $Y\subseteq X$, ${\rm rank}\big(T((z_n),Y,K)\big)>\alpha$.
\end{lemme}

\begin{proof}
Let $Y\subseteq X$ and suppose toward a contradiction that ${\rm rank}\big(T((z_n),Y,K)\big)=\xi_0+1\leqslant\alpha$, where $\xi_0$ is the rank of the root $\tom$ in $T((z_n),Y,K)$. Now, let I play $Y,\xi_0$ in $G^\alpha_X$ and let II respond using his strategy
$$
{\footnotesize 
\begin{array}{cccccccccccc}
{\bf I}  & & Y      &   	                 &            \\
           &&\xi_0     & 		        &                                         \\
           &&              &                        &            \\
{\bf II}  &  &           &E_0\subseteq Y         &              \\
            &    &         & y_0\in E_0    &         
 \end{array}
}
$$
Then the rank of $(y_0)\in T((z_n),Y,K)$ is some ordinal $\xi_1<\xi_0$, so in $G^\alpha_X$, I continues by playing $Y,\xi_1$ and II responds according to his strategy
$$
{\footnotesize 
\begin{array}{cccccccccccc}
{\bf I}  & & Y      &   	                 &  Y          \\
           &&\xi_0     &                 &  \xi_1                                       \\
           &&              &                        &            &                  \\
{\bf II}  &  &           &E_0\subseteq Y         &        & E_1\subseteq Y      \\
            &    &         & y_0\in E_0    &                      & y_1\in E_0+E_1
 \end{array}
}
$$
Again,  the rank of $(y_0,y_1)\in T((x_n),Y,K)$ is some ordinal $\xi_2<\xi_1$, so in $G^\alpha_X$, I continues by playing $Y,\xi_2$ and II responds according to his strategy
$$
{\footnotesize 
\begin{array}{cccccccccccc}
{\bf I}  & & Y      &   	                 &  Y         && Y \\
           &&\xi_0     &                 &  \xi_1        && \xi_2                               \\
           &&              &                        &            &                  \\
{\bf II}  &  &           &E_0\subseteq Y         &        & E_1\subseteq Y  &&E_2\subseteq Y    \\
            &    &         & y_0\in E_0    &                      & y_1\in E_0+E_1&& y_2\in E_0+E_1+E_2
 \end{array}
}
$$
Etc.

Eventually, we will have constructed some $(y_0,y_1,\ldots,y_{k-1})$ whose $T((z_n),Y,K)$-rank  is $\xi_{k}=0$, while 
$$
{\footnotesize 
\begin{array}{cccccccccccc}
{\bf I}   & Y      &   	                 &         &               &&Y \\
           &\xi_0    &                 &        &  && \xi_{k-1}                               \\
           &              &                        &             &\ldots                \\
{\bf II}   &           &E_0\subseteq Y         &        &         &&&E_{k-1}\subseteq Y    \\
            &         & y_0\in E_0    &                      & &              &&y_{k-1}\in E_0+\ldots+E_{k-1}
 \end{array}
}
$$
has been played according to the strategy of II.

It follows that if I continues the game by playing $Y,\xi_k=0$, 
$$
{\footnotesize 
\begin{array}{cccccccccccc}
{\bf I}   & Y      &   	                 &         &               &&Y&&Y \\
           &\xi_0    &                 &        &  && \xi_{k-1}                  &&\xi_k=0             \\
           &              &                        &             &\ldots               && \\
{\bf II}   &           &E_0\subseteq Y         &        &         &&&E_{k-1}\subseteq Y  &&  \\
            &         & y_0\in E_0    &                      & &              &&y_{k-1}\in E_0+\ldots+E_{k-1}&&
 \end{array}
}
$$
using his strategy, II must be able to respond with some $E_k$ and $y_k\in E_0+\ldots+E_k$
$$
{\footnotesize 
\begin{array}{cccccccccccc}
{\bf I}   & Y      &   	                                         &&Y&&Y \\
           &\xi_0    &                 &          & \xi_{k-1}                  &&\xi_k=0             \\
           &              &                                     &\ldots               & \\
{\bf II}   &           &E_0\subseteq Y                 &         &&E_{k-1}\subseteq Y  && E_{k}\subseteq Y  \\
            &         & y_0\in E_0                          & &              &y_{k-1}\in E_0+\ldots+E_{k-1}&&y_{k}\in E_0+\ldots+E_{k}
 \end{array}
}
$$
Since II played according to his strategy, we have $(y_0,y_1,\ldots, y_{k})\sim_K(z_0,z_1,\ldots,z_k)$ and thus  $(y_0,y_1,\ldots, y_{k})\in T\big((z_n),Y,K\big)$, contradicting that 
$(y_0,\ldots,y_{k-1})$ has $T\big((z_n),Y,K\big)$-rank $0$ and hence is a terminal node.
\end{proof}

\begin{lemme}\label{tight}
Suppose $(x_n)\subseteq W$ is a block sequence, $\beta<\om_1$,  and that for every normalised block sequence $(y_n)$ in $X=[x_n]$ and $K\geqslant 1$, I has a strategy in $F^{\beta}_X$ to play $(z_0,z_1,\ldots,z_k)$ such that
$$
(z_0,z_1,\ldots,z_k)\not\sim_K(y_0,y_1,\ldots,y_k).
$$
Then, for every normalised block sequence $(y_n)$ in $X$ and  $K\geqslant 1$, there is a sequence $(J_m)$ of intervals of $\N$ with $\min J_m\til \infty$, such that if $A\subseteq \N$ is infinite, contains $0$ and $Z=[x_j\del j\notin \bigcup_{m\in A}J_m]$, then 
$$
{\rm rank}\big(T((y_n),Z,K)\big)\leqslant \beta.
$$ 
\end{lemme}

\begin{proof}We relativise the notions of support of vectors et cetera to the basis $(x_n)$ for $X$. So the reader can assume that $(x_n)$ is the original basis $(e_n)$ and $X=W$.

Assume $(y_n)$ is a normalised block sequence in $X$ and $K\geqslant 1$. Let also $\Delta=(\delta_j)$ be a sequence of positive real numbers such that whenever $z_j,v_j\in X$, $\|z_j-v_j\|<\delta_j$, and 
$$
(v_0,\ldots,v_k)\sim_K(y_0,\ldots,y_k),
$$
then 
$$
(z_0,\ldots,z_k)\sim_{2K}(y_0,\ldots,y_k).
$$
We choose sets
$\D_i\subseteq  X$  such that  for each finite set $d\subseteq \N$, the number
of $z\in \D_i$ such that ${\rm supp}(z)=d$ is finite, and for every $v\in 
X$ with $\|v\|\leqslant K$ there is some $z\in \D_i$ with ${\rm supp}(z)={\rm supp}(v)$ and $\norm{z-v}<\delta_i$. This is possible since the $K$-ball in $[x_j]_{j\in d}$
is totally bounded for all finite $d\subseteq\N$.

The strategy for I in $F^\beta_X$ in the game for $(y_n)$ with constant $2K$ can be seen as a pair of functions $\xi$ and $n$ that to each legal position $(z_0,E_0,\ldots,z_j,E_j)$ of II in $F^\beta_X$ provide the next play $\xi(z_0,E_0,\ldots,z_j,E_j)\in {\rm Ord}$ and $n(z_0,E_0,\ldots,z_j,E_j)\in \N$ by I. 

We define a function $p\colon \N\til \N$ by letting $p(m)$ be the maximum of $m$ and 
\begin{align*}
\max\big(n(z_0,[x_l]_{l\in d_0},\ldots,z_i,[x_l]_{l\in d_i})\del d_j\subseteq [0,m-1]\;&\: z_j\in [x_l]_{l\in d_0\cup\ldots\cup d_j}\cap \D_j\big).
\end{align*}
By assumption on the sets $\D_j$, $p$ is well-defined and so we can set $J_m=[m,p(m)]\subseteq \N$.

We claim that if $A\subseteq \N$ is an infinite set containing $0$ and 
$$
Z=[x_n\;|\;n\notin \bigcup_{m\in A}J_m],
$$ 
then
$$
{\rm rank}\big(T((y_n),Z,K)\big)\leqslant \beta.
$$
To see this,  we define a monotone function $\phi$, i.e., $\vec v\prec \vec w\saa \phi(\vec v)\prec \phi(\vec w)$, associating to each  $\vec v=(v_0,v_1,\ldots,v_{i})\in T((y_n),Z,K)$ some 
$$
\phi(\vec v)=(z_0,z_1,\ldots,z_{i})\in \D_0\times \D_1\times\ldots\times \D_{i}
$$ 
such that for all $j\leqslant i$, $\|z_j-v_j\|<\delta_j$ and ${\rm supp}(z_j)={\rm supp}(v_j)$, whereby, in particular, $z_j\in Z$. Also set $T=\phi\big[T((y_n),Z,K)\big]$ and note that $T$ is a subtree of $Z^{<\N}$ with 
$$
{\rm rank}(T)\geqslant {\rm rank}\big(T((y_n),Z,K)\big).
$$

Suppose toward a contradiction that ${\rm rank}(T)>\beta$, whereby the rank of $\tom$ in $T$ is $\geqslant \beta$. We describe how II can play against the strategy for I in $F^\beta_X$ to play $(z_0,\ldots,z_k)$ such that
$$
(z_0,\ldots,z_k)\sim_{2K}(y_0,\ldots,y_k),
$$
which will contradict the assumption on the strategy for I. The case $\beta=0$ is trivial, so we assume that $\beta>0$.

First, I plays $\xi(\tom)<\beta$ and $n(\tom)$. Since, $a_0=0\in A$, we have $n(\tom)\leqslant p(a_0)=\max J_{a_0}<Z$ and thus there is some $n(\tom)<z_0\in T$ whose rank in $T$ is $\geqslant \xi(\tom)$.
Find also $a_1\in A$ such that $z_0<J_{a_1}$ and let $E_0=[x_j\del J_{a_0}<x_j<J_{a_1}]$.
So let II respond by 
$$
{\footnotesize 
\begin{array}{cccccccccccc}
{\bf I}  & & n(\tom)      &   	                 &            \\
           &&\xi(\tom)     & 		        &                                         \\
           &&              &                        &            \\
{\bf II}  &  &           &n(\tom)<E_0         &              \\
            &    &         & z_0\in E_0    &         
 \end{array}
}
$$
Now, by his strategy, I will play some $\xi(z_0,E_0)<\xi(\tom)$ and $n(z_0,E_0)\leqslant p(a_1)=\max J_{a_1}$. So find some $z_1$ such that $(z_0,z_1)\in T$ and has rank $\geqslant \xi(z_0,E_0)$ in $T$. Find also $a_2\in A$ such that $z_1<J_{a_2}$. Then, as $a_0,a_1\in A$, if we set $E_1=[x_j\del J_{a_1}<x_j<J_{a_2}]$, we have $z_1\in E_0+E_1$, so we let II respond by 
$$
{\footnotesize 
\begin{array}{cccccccccccc}
{\bf I}  & & n(\tom)      &   	                 &   n(z_0,E_0)         \\
           &&\xi(\tom)     & 		        &                 \xi(z_0,E_0)                        \\
           &&              &                        &            \\
{\bf II}  &  &           &n(\tom)<E_0         &    &n(z_0,E_0)<E_1          \\
            &    &         & z_0\in E_0    &         &z_1\in E_0+E_1
 \end{array}
}
$$
Et cetera. 
It follows that at the end of the game, 
$$
{\footnotesize 
\begin{array}{cccccccccccc}
{\bf I}   & n(\tom)      &   	&          &   n(z_0,E_0, \ldots, z_{k-1},E_{k-1})         \\
           &\xi(\tom)     & &		        &               \xi(z_0,E_0, \ldots, z_{k-1},E_{k-1})=0                      \\
           &              &          &   \ldots           &  &          \\
{\bf II}   &           &n(\tom)<E_0         &&    &n(z_0,E_0, \ldots, z_{k-1},E_{k-1}) <E_k          \\
               &         & z_0\in E_0    &      &   &z_k\in E_0+\ldots+E_k
 \end{array}
}
$$
II will have constructed a sequence $(z_0,\ldots,z_k)\in T$. So, by the definition of $T$, there is some $(v_0,\ldots,v_k)\in T((y_n),Z,K)$ such that $\phi(v_0,\ldots,v_k)=(z_0,\ldots,z_k)$ and hence $\|z_j-v_j\|<\delta_j$ for all $j$. Thus,
$$
(v_0,\ldots,v_k)\sim_K(y_0,\ldots,y_k), 
$$
and hence
$$
(z_0,\ldots,z_k)\sim_{2K}(y_0,\ldots,y_k).
$$
Since II cannot have such a strategy, it follows instead that 
$$
{\rm rank}\big(T((y_n),Z,K)\big)\leqslant {\rm rank}(T)\leqslant \beta,
$$
which proves the lemma.
\end{proof}

\begin{lemme}\label{tight1}
Suppose $(x_n)\subseteq W$ is a normalised block sequence, $\beta<\om_1$,  and that for every normalised block sequence $(y_n)$ in $X=[x_n]$ and $K\geqslant 1$, I has a strategy in $F^{\beta}_X$ to play $(z_0,z_1,\ldots,z_k)$ such that
$$
(z_0,z_1,\ldots,z_k)\not\sim_K(y_0,y_1,\ldots,y_k).
$$
Then, for every normalised block sequence $(y_n)$ in $X$ there is a sequence 
$$
I_0<I_1<I_2<\ldots
$$ 
of intervals of $\N$, such that if $A\subseteq \N$ is infinite and $Z=[x_j\del j\notin \bigcup_{m\in A}I_m]$, then 
$$
{\rm Emb}\big((y_n),Z\big)\leqslant \beta.
$$ 
\end{lemme}

\begin{proof}Fix a normalised block sequence $(y_n)$ in $X$ and relativise again all notions of support et cetera to the block basis $(x_n)$.
By Lemma \ref{tight}, we can for every $K$ find a sequence $(J_n^K)$ of intervals of $\N$ with $\min J_n^K\Lim{n\til \infty} \infty$ such that for any infinite set $A\subseteq \N$ containing $0$, we have
$$
{\rm rank}\big(T((y_n),[x_j\del j\notin\bigcup_{n\in A}J^K_n],K)\big)\leqslant \beta.
$$ 
Also, for every $N$, we let $c(N)\in \N$ be a constant such that any two subsequences of $(x_j)$ differing in at most $N$ terms are $c(N)$-equivalent. 

We construct intervals $I_0<I_1<I_2<\ldots$ such that each $I_n$ contains an interval from each of the families $(J^1_i),\ldots,(J^n_i)$ and, moreover,  
$$
\min I_n<\max I_n-\max J^{n\cdot c(\min I_n)}_0.
$$

We claim that  if $A\subseteq \N$ is infinite and $Z=[x_j\del j\notin \bigcup_{m\in A}I_m]$, then 
$$
{\rm Emb}\big((y_n),Z\big)\leqslant \beta.
$$ 
Suppose towards a contradiction that this fails for some $A$ and pick some $N$ such that ${\rm rank}\big(T((y_n),Z,N)\big)>\beta$. Choose $a\in A$ such that $a\geqslant N$ and note that
$$
\min I_a<\max I_a-\max J^{a\cdot c(\min I_a)}_0.
$$
Thus, by changing only the terms $x_j$ for $j<\min I_a$ of the sequence 
\[\begin{split}
(x_j\del j\notin \bigcup_{m\in A}&I_m)=\\
& (x_j\del j\notin \bigcup_{m\in A}I_m\;\&\; j<\min I_a)\cup (x_j\del j\notin \bigcup_{m\in A}I_m\;\&\; j>\max I_a),
\end{split}\]
we find a  subsequence of 
\[\begin{split}
(x_j\del \max J_0^{a\cdot c(\min I_a)}<j\leqslant \max I_a)\cup 
(x_j\del j\notin \bigcup_{m\in A}I_m\;\&\; j>\max I_a)\end{split}\]
that is $c(\min I_a)$-equivalent with
$$
(x_j\del j\notin \bigcup_{m\in A}I_m).
$$
Since $N\cdot c(\min I_a)\leqslant a\cdot c(\min I_a)$, it follows that if
$$
Y= [x_j\del \max J_0^{a\cdot c(\min I_a)}<j\leqslant \max I_a]+[x_j\del j\notin \bigcup_{m\in A}I_m\;\&\; j>\max I_a],
$$
then $Z\sqsubseteq_{c(\min I_a)}Y$, and so
$$
\beta<{\rm rank}\big(T((y_n),Z,N)\big)\leqslant {\rm rank}\big(T((y_n),Y,a\cdot c(\min I_a))\big).
$$ 
But, by the choice of the $I_n$, we see that there is an infinite subset $B\subseteq \N$ containing $0$ such that $Y$ is outright a subspace of $[x_j\del j\notin \bigcup_{m\in B}J^{a\cdot c(\min I_a)}_m]$, whereby, by choice of the intervals $J^{a\cdot c(\min I_a)}_m$, we have
$$
{\rm rank}\big(T((y_n),Y,a\cdot c(\min I_a))\big)\leqslant  \beta,
$$
which is absurd.
This contradiction shows that the intervals $I_n$ fulfill the conclusion of the lemma.
\end{proof}

By combining Theorem \ref{main1} and Lemmas \ref{loose} and \ref{tight1}, we obtain
\begin{thm}\label{main3}
Suppose $\alpha<\om_1$. Then there is a block subspace  $X=[x_n]\subseteq W$ such that one of the following holds
\begin{enumerate}
\item For every normalised block sequence $(y_n)$ in $X$ there is a sequence 
$$
I_0<I_1<I_2<\ldots
$$ 
of intervals of $\N$, such that if $A\subseteq \N$ is infinite, then 
$$
{\rm Emb}\big((y_n),[x_j\del j\notin \bigcup_{m\in A}I_m]\big)\leqslant \om\alpha.
$$ 
\item For any subspace $Y\subseteq X$ and any block sequence $(z_n)\subseteq X$,
$$
{\rm Emb}\big((z_n),Y\big)> \alpha.
$$
\end{enumerate}
\end{thm}

And by replacing the normed $\go F$-vector subspaces $X$ and $Y$ in Theorem \ref{main3} by their closures $\ku X$ and $\ku Y$ in $\ku W$, we obtain
Theorem \ref{main5}.
\begin{thm}
Let $\ku W$ be Banach space with a Schauder basis and suppose $\alpha<\om_1$. 
Then there is a block subspace  $\ku X=[x_n]\subseteq \ku W$ that is either $\om\alpha$-tight or $(\alpha+1)$-minimal.
\end{thm}

\end{document}